\documentclass[letter,final]{amsart}

\usepackage{amsmath}%
\usepackage{amsfonts}%
\usepackage{amssymb}%
\usepackage{amsthm}
\usepackage{graphicx}
\usepackage{dsfont}
\usepackage{pdfcomment}
\usepackage{tikz}
\usepackage{tikz-3dplot}
\usepackage{subfigure}
\usepackage{soul}
\usepackage{enumerate}
\usepackage{cancel}

\usepackage{hyperref}
\usepackage{cleveref}
\usepackage{esvect}

\usetikzlibrary{decorations.markings}
\newtheorem{theorem}{Theorem}[section]

\newtheorem{corollary}[theorem]{Corollary}

\newtheorem{definition}[theorem]{Definition}
\newtheorem{example}[theorem]{Example}
\allowdisplaybreaks

\newtheorem{proposition}[theorem]{Proposition}
\newtheorem{remark}[theorem]{Remark}

\DeclareMathOperator{\Real}{Re}
\DeclareMathOperator{\Imaginary}{Im}

\DeclareMathOperator{\area}{Area}

\renewcommand{\Re}{\Real}
\renewcommand{\Im}{\Imaginary}

\title[Weighted toroidal graphs]{Delaunay decompositions minimizing energy of weighted toroidal graphs}

\author{Wai Yeung Lam}

\address{Department of Mathematics, University of Luxembourg, Maison du nombre, 6 avenue de la Fonte, L-4364 Esch-sur-Alzette, Luxembourg.} \email{wyeunglam@gmail.com}
\thanks{The author was partially supported by the FNR grant CoSH O20/14766753 and the Institute of Mathematical Sciences at The Chinese University of Hong Kong.}


\begin{document}
	\maketitle
	
	\begin{abstract}		
		Given a weighted graph on a torus, each realization to a Euclidean torus is associated with the Dirichlet energy. By minimizing the energy over all possible Euclidean structures and over all realizations within a fixed homotopy class, one obtains a harmonic map into an optimal Euclidean torus. We show that only with this optimal Euclidean structure, the harmonic map and the edge weights are induced from a weighted Delaunay decomposition. 
	\end{abstract}
	
	\section{Introduction}
	
	Let $G=(V,E,F)$ be a cellular decomposition of a topological torus $S$ where $V$, $E$ and $F$ denote the set of vertices, edges and faces. It induces a realization $f:(V,E) \to S$ of the 1-skeleton graph $(V,E)$. We assume that all faces have at least $3$ edges and the graph is \emph{essentially 3-connected}, i.e. the 1-skeleton graph of the induced cellular decomposition on the universal cover is 3-connected. We equip the graph $(V,E)$ with some positive edge weights $c:E \to \mathbb{R}_{>0}$ where $c_{ij}=c_{ji}$. We relate three aspects of embedding a weighted graph in Euclidean tori. \medskip
	
	\noindent \textbf{Energy minimization.} We parameterize the space of Euclidean tori with unit area by the upper half plane $\mathbb{H} \subset \mathbb{C}$ and write $S_{\tau}$ for the corresponding Euclidean torus given by $\tau \in \mathbb{H}$. For every straight-line mapping to a Euclidean torus $h:(V,E) \to S_{\tau}$, it is associated with the Dirichlet energy
	\begin{align}\label{eq:energy}
		D_c(h) := \frac{1}{2} \sum_{ij \in E} c_{ij} \ell_{ij}^2.
	\end{align}
    where $\ell$ is the edge length. There are several interpretations of the edge weights and the energy. One is to think of the edges as springs with force constant $c$. The Dirichlet energy is then the total energy stored in the springs.

	It is a classical result that for every fixed Euclidean torus $S_\tau$ and among all mappings $h:(V,E) \to S_{\tau}$ such that $h$ is homotopic to $f$, the energy $D_c$ has a minimizer $h_{\tau}$ unique up to translations \cite{CdV1991,Gortler2006,Lovasz2004}. The map $h_\tau$ has convex faces and is called a Tutte-like embedding \cite{Gortler2006}. It is known to be a harmonic map satisfying a discrete Laplace's equation (See Section \ref{sec:background}).

	One can further minimize the energy by varying Euclidean structures, i.e. to consider the function 
	\begin{align*}
			\mathcal{D}_c:  \mathbb{H} \times \mathcal{C}_f &\to \mathbb{R} \\
			                           (\tau, h) &\mapsto  D_c(h)
	\end{align*}
   where $ \mathcal{C}_f $ denotes the space of mappings $h$ that are homotopic to $f$. It is known that $\mathcal{D}_c$ has a unique minimizer $(\tau,h_{\tau})$ \cite{Kotani2001}. Our goal is to explore the geometric relation how the optimal Euclidean structure depends on the edge weights. For example, we shall answer a converse question. \smallskip
   
\noindent   \textbf{Question A:}
   Suppose a straight-line embedding $h:V \to S_{\tau}$ to a Euclidean torus $S_\tau$ is given with convex faces. If any, what are the positive edge weights $c$ such that $(\tau,h)$ is the minimizer of $\mathcal{D}_c$? \medskip
   
  \noindent \textbf{Maxwell-Cremona correspondence.} Every harmonic map to $\mathbb{R}^2$ is associated with a realization of the dual graph such that corresponding edges are perpendicular. A question is whether it holds for the torus. Specifically, for every harmonic map $h:(V,E) \to S_{\tau}$, its lift to the universal cover $\tilde{h}:(\tilde{V},\tilde{E}) \to \mathbb{C}$ has a conjugate map $\tilde{h}^*: \tilde{F} \to \mathbb{C}$ defined by
   \begin{align}\label{eq:conharmmap}
   	 \tilde{h}^*_{\mbox{left}(ij)} -\tilde{h}^*_{\mbox{right}(ij)} =  \sqrt{-1} \, c_{ij} (\tilde{h}_j -\tilde{h}_i) 
   \end{align}
   where $\mbox{left}(ij)$ is the left face of the oriented edge $ij$. 
   The map $\tilde{h}^*$ defines a realization of the dual decomposition of the universal cover $(\tilde{V}^*,\tilde{E}^*,\tilde{F}^*)\cong  (\tilde{F},\tilde{E},\tilde{V})$. It is harmonic with respect to the edge weights $c^{*}:= 1/c$. The edges of $\tilde{h}^*$ are orthogonal to those of $\tilde{h}$. The maps $\tilde{h}$ and $\tilde{h}^*$ are said to be reciprocal in the sense of the Maxwell–Cremona correspondence \cite{Maxwell1870}. Since $\tilde{h}$ is doubly periodic, one can deduce that $\tilde{h}^{*}$ projects to an Euclidean torus $k S_{\tau'}$ for some $k>0$ and $\tau'\in \mathbb{H}$. Here $k S_{\tau'}$ denotes the scaled copy of the Euclidean torus $S_{\tau'}$ having area $k^2$.  \medskip

   \noindent   \textbf{Question B:}
   Given an edge weight $c$ and any Euclidean torus $S_\tau$, the harmonic map $h_{\tau}:(V,E) \to S_{\tau}$ defines a conjugate harmonic map $\tilde{h}^{*}:(\tilde{V}^*,\tilde{E}^*) \to \mathbb{C}$ over the universal cover. Does $\tilde{h}^{*}$ project to the same torus $S_{\tau}$ up to scaling? It is shown in \cite{Lin2020} that this holds only for some unique Euclidean structure. \medskip

\noindent \textbf{Edge weights from weighted Delaunay decomposition.} The previous two aspects concern geometric embeddings from given edge weights. A converse question is to deduce edge weights from a given embedding.

A weighted Delaunay decomposition is also known as a power diagram. It is a generalization of the classical Delaunay decomposition where the distance function are modified by vertex weights. Its dual cell decomposition is the Voronoi diagram. One can show that the dual edge is always perpendicular to the primal edge. In case where the vertex weights are constant, one obtains the classical Delaunay decomposition. 

 Given a Delaunay decomposition of the plane or a Euclidean torus, it is associated with a canonical choice of edge weights called the \textit{cotangent weights}  \cite{Pinkall1993}. It is related to discrete complex analysis and deformations of circle packing. More generally, given a weighted Delaunay decomposition $h:V \to \mathbb{C}$, we denote $h^{\dagger}$ its dual Voronoi diagram. The \textit{canonical edge weight} $c:E \to \mathbb{R}_{>0}$ is given by
\begin{equation}\label{eq:dualprimal}
  c_{ij}=  \frac{1}{\sqrt{-1}}	\frac{ \tilde{h}^{\dagger}_{\mbox{left}(ij)} -\tilde{h}^{\dagger}_{\mbox{right}(ij)} }{ \tilde{h}_j -\tilde{h}_i }  =c_{ji} >0.
\end{equation}
A question is whether such a correspondence from geometric embeddings to edge weights provides the converse relation in the previous two aspects. We show that it is the case.

	  \begin{theorem}\label{thm}
	Let $(V,E,F)$ be a cellular decomposition of a topological torus and essentially 3-connected. Assume $c:E \to \mathbb{R}_{>0}$ be positive edge weights. Then the following statements on Euclidean structure $\tau$, discrete harmonic map $h_{\tau}$ as well as a positive number $k$ are equivalent: 
	\begin{enumerate}[(i)]
		\item $(\tau, h_{\tau})$ is the unique minimizer of $\mathcal{D}_c$ with value $k$.
		\item The conjugate map of harmonic map $h_{\tau}:V \to S_{\tau}$ projects to the same torus $S_{\tau}$ up to scaling $k>0$.
		\item $h_{\tau}$ is the weighted Delaunay decomposition such that the given weight $c$ is k-multiple of the canonical edge weight.
	\end{enumerate}
 We denote $\tau_c$ the unique Euclidean structure satisfying the conditions above and call it the optimal Euclidean structure, whose energy is written as $k_c$.
\end{theorem}

As shown in our proof, the equivalence between (i) and (ii) holds for any cellular decomposition of a torus, without the assumption of being essentially 3-connected. However, the graph being essentially 3-connected is a necessary condition for (iii) to make sense, since the graph of any weighted Delaunay decomposition is essentially 3-connected (See the proof of Balinski's theorem \cite{Ziegler}). With this assumption, the relation between (ii) and (iii) is known to Erickson and Lin \cite[Theorem 4.4]{Lin2020}. Our main contribution is their connection to (i). Our approach involves studying discrete harmonic conjugates, which is also known as the response matrix in electric networks \cite{Kenyon2015}.

On the other hand, one observes that the minimal Dirichlet energy $k_c$ is intrinsic to the weighted toroidal graph. One should be able to express $k_c$ in terms of the edge weights without involving the embedding $h$ as in equation \eqref{eq:energy}. We first state the formula and then explain the notations.
 
\begin{theorem}\label{thm:energy}
		Let $c:E \to \mathbb{R}_{>0}$ be a positive edge weight. Then the minimal Dirichlet energy is given by
	\[
	k_c := \min \mathcal{D}_c =  \sqrt{ \frac{\det_0 (\tilde{d}^T C \tilde{d})}{\det_0( d^T C d)}}
	\]
\end{theorem}
In the formula, we fix an arbitrary orientation for the edges and $d$ denotes the $E{\times}V$-incidence matrix, where $d^T$ is its transpose. $C$ is the $E{\times}E$ diagonal matrix consisting of edge weights. Thus, $\Delta:=d^T C d$ is a $V{\times}V$-matrix that represents the discrete Laplace operator, i.e. for every $g:V \to \mathbb{R}$,
\[
(\Delta g)_i = \sum_{j} c_{ij}(g_j-g_i). 
\]
Furthermore, pick a vertex $o$. Then $\det_0$ denotes the determinant of a matrix with the row and column corresponding to the vertex $o$ removed. In our case the quantity is independent of the vertex chosen. On the other hand, $\tilde{d}$ is a $E{\times}(V{+}2)$-matrix. Its first $V$ columns are those of $d$, which span the space of exact 1-forms. The last two columns of $\tilde{d}$ represent closed 1-forms with nontrivial periods, whose integrals along $\gamma_1$ and $\gamma_2$ are respectively $(1,0)$ and $(0,1)$. Particularly, the formula in Theorem \ref{thm:energy} suggests the minimal energy is related to a combinatorial problem that counts spanning trees in the graph via Kirchhoff's matrix-tree theorem.

More generally, Theorem \ref{thm} and Theorem \ref{thm:energy} hold for edge weights with arbitrary signs as long as the energy functional is positive definite (See Section \ref{sec:sign}). 

\subsection{Related work}

It is common in discrete conformal geometry \cite{Glickenstein2007,Glickenstein2011,Kenyon2002,Lam2019,Mercat2001,Springborn2008} and computer graphics \cite{Gortler2006,Gu2003} to consider edge weights in the form of equation \eqref{eq:dualprimal} from a Delaunay decomposition. On the one hand, the corresponding discrete harmonic functions describe infinitesimal deformations of circle packings \cite{Lam2015a}. On the other hand, it includes the cotangent-weight Laplacian which is obtained from the finite element discretization \cite{Pinkall1993}. Compared to them, our result is about the converse construction from a weighted graph to a Delaunay decomposition.

Delaunay decompositions minimize energy in several other ways. An example is Rippa's theorem \cite{Rippa1990}. For that setting, vertex positions are fixed while combinatorics are allowed to change. It is in contrast to our setting where combinatorics and edge weights are fixed but vertex positions are allowed to move.

Realizations of graphs to other surfaces that minimize the Dirichlet energy have been considered \cite{Urschel2021}. For a finite planar graph, one obtains the classical Tutte embedding. For hyperbolic surfaces, one also gets a harmonic map to an optimal hyperbolic surface \cite{Toru2021} and Theorem \ref{thm} holds analogously \cite{lam2024discrete}.

Our approach involves discrete harmonic 1-forms and their harmonic conjugate. Although we apply it to Euclidean tori, the method is also useful for complex affine tori \cite{lam2024space}. Similar questions can be asked for surfaces of genus larger than 1 and are related to translation surfaces. However, its connection to Delaunay decomposition is unclear. 

As seen in equation \eqref{eq:dualprimal}, we took the perspective where the edge weight is the ratio of the dual edge length to the primal edge length. It is similar to conductance in electric networks where conductance is the ratio of the current to the voltage difference. With different interpretation of the edge weights, one obtains different geometric realizations of the graph. For example, if one regard the edge weights as distance between circumcenters, one also can obtain circle patterns \cite{dimer2018}.

We also consider edge weights with arbitrary signs under the condition that the energy functional is positive definite (Definition \ref{def:energy}). Nevertheless, it is interesting to consider cases where the energy functional fails to be positive definite (See \cite{Borcea2015,Lam2015a} for example).

\section{The space of marked Euclidean tori}\label{sec:background}

We denote $G=(V, E, F)$ a cellular decomposition of a topological torus where $V$, $E$ and $F$ are the sets of vertices, edges and faces respectively. Vertices are denoted by $i,j,k$. An edge is denoted by $ij$ indicating its end points are vertices $i$ and $ j$, where $i=j$ is allowed and in that case the edge has to form a non-contractible loop on the surface. Furthermore, we denote $\tilde{G}=(\tilde{V},\tilde{E},\tilde{F})$ the lift of the cellular decomposition to the universal cover.

	We assume that $\gamma_1,\gamma_2$ are the generators of the fundamental group. We parameterize the space of Euclidean tori with unit area by the upper half plane $\mathbb{H} \subset \mathbb{C}$
such that for each $\tau \in \mathbb{H}$, the Euclidean torus $S_\tau$ is obtained as a quotient of $\mathbb{R}^2 \cong \mathbb{C}$ by translations
\[
\rho_{\gamma_1}(z)= z + \frac{1}{\sqrt{\Im \tau}}, \quad  	\rho_{\gamma_2}(z)= z + \frac{\tau}{\sqrt{\Im \tau}}.
\]

For every mapping $h:(V,E) \to S_{\tau}$ where edges are realized as straight lines, we consider its lift to the universal cover $\tilde{h}:(\tilde{V},\tilde{E}) \to \mathbb{C}$. It satisfies
\begin{equation}\label{eq:htperiods}
	\tilde{h}\circ \gamma_1=  \tilde{h} + \frac{1}{\sqrt{\Im \tau}}, \quad \tilde{h}\circ \gamma_2=  \tilde{h} +\frac{\tau}{\sqrt{\Im \tau}}
\end{equation}
In terms of the lift, the energy can be expressed as
\[
D_c(h) := \frac{1}{2} \sum_{ij \in E} c_{ij} |\tilde{h}_{\tilde{j}} - \tilde{h}_{\tilde{i}}|^2.
\]
where $\widetilde{ij} \in \tilde{E}$ is a lift of edge $ij$. The map $h$ is harmonic if it is a critical point of the energy under variation of vertex positions. Equivalently it satisfies discrete Laplace's equation, i.e. for $i \in \tilde{V}$
\begin{equation}\label{eq:harmonic}
\sum_j c_{ij} (\tilde{h}_j - \tilde{h}_i) =0.
\end{equation}
where the summation is over the adjacent edges $ij \in \tilde{E}$. In the following sections, we shall interpret $\Re \tilde{h}$ and $\Im \tilde{h}$ as integrals of harmonic 1-forms with specific periods.

\section{Discrete harmonic 1-forms}

A discrete 1-form is a function on oriented edges $\omega: \vec{E} \to \mathbb{R}$ such that $\omega_{ij}=-\omega_{ji}$ for every edge $ij$. A discrete 1-form is closed if its summation over the boundary of every oriented face is zero. For example, in the case of a triangulation, $\omega$ is closed if for every triangle $\{ijk\}$,
\[
\omega_{ij} +\omega_{jk}+ \omega_{ki} =0.
\]
For a closed discrete 1-form $\omega$ on a torus, one can consider its periods
\[
\sum_{\gamma_1} \omega = A, \quad \sum_{\gamma_2} \omega = B 
\]
where the summation is over an edge path homotopic to $\gamma_k$. Because $\omega$ is closed, the summation is independent of the path chosen. A 1-form $\omega$ is exact if there exists $f:V \to \mathbb{R}$ such that
\[
\omega_{ij}= f_j-f_i.
\] 
One can show that a 1-form is exact if and only if it is closed with vanishing periods, i.e. $(A,B)=(0,0)$ in the case of tori.

The orientation of the primal edges naturally induces an orientation for the dual edges. Given an oriented edge $ij$, the dual edge $*ij$ is oriented from right face of $ij$ to the left face. In this way, we say a 1-form $\omega$ is co-closed if it is a closed 1-form with respect to the dual decomposition,  i.e. for every vertex $i\in V$
\[
\sum_j \omega_{ij} =0
\]
where the summation is over all edges adjacent to $i$.

For every 1-form $\omega$, there is an associated 1-form $*\omega$ defined by 
\[
(*\omega)_{ij}:=c_{ij} \omega_{ij} .\] The map sending $\omega$ to $*\omega$ is a discrete analogue of the Hodge star operator.

We call  $\omega$ a harmonic 1-form on the primal decomposition $(V,E,F)$ if $\omega$ is closed and $*\omega$ is co-closed. The co-closeness implies for every vertex $i \in V$
\[	\sum_{j} (*\omega)_{ij} = \sum_j c_{ij} \omega_{ij} =0. \]
One can check that $*\omega$ is a harmonic 1-form with respect to the dual cell decomposition and edge weights $c^*:=1/c$. We also call $*\omega$ the harmonic conjugate of $\omega$.
 
 It is known \cite[Theorem 3.9]{bobenko2016} that the space of discrete harmonic 1-forms on a torus is isomorphic to $\mathbb{R}^2$ and is parameterized by the period $(A,B) \in \mathbb{R}^{2}$. 
	
\section{Response matrix over the period space} \label{sec:conjugate}

For any edge weights $c$, we consider the response matrix on the period space
	\begin{align*}
	L: \mathbb{R}^{2} &\to \mathbb{R}^{2} \\
	(A, B) & \mapsto (A^*, B^*)
\end{align*}
which maps the periods of a harmonic 1-form on $(V,E,F)$ to the periods of the conjugate harmonic 1-form. Namely, for every $(A,B) \in \mathbb{R}^2$, there exists a unique harmonic 1-form $\omega : \vec{E} \to \mathbb{R}$ such that
\[
\sum_{\gamma_1} \omega = A, \quad \sum_{\gamma_2} \omega = B. 
\]
Its conjugate would have periods
\[
\sum_{\gamma_1} *\omega = A^*, \quad \sum_{\gamma_2} *\omega = B^*. 
\]
Then we define $L(A,B)=(A^*,B^*)$. We shall relate it to the Dirichlet energy and derive an explicit form of $L$.

\begin{definition}
	We define a skew symmetric bilinear form over $\mathbb{R}^2$. For any $(A,B),(\tilde{A},\tilde{B}) \in \mathbb{R}^2$,
	\[
	\{(A,B),(\tilde{A},\tilde{B})\}:= A\tilde{B} - B \tilde{A}.
	\]
\end{definition}
   	One can check that the bilinear form is non-degenerate in the sense that $(A,B) \in \mathbb{R}^2$ satisfies 
   \[
   \{(A,B),(\tilde{A},\tilde{B})\} = 0 \quad \forall (\tilde{A},\tilde{B}) \in \mathbb{R}^2
   \]
   if and only if $(A,B)=(0,0)$.

Given a closed 1-form $\omega$ and a co-closed 1-form $\tilde{\omega}$ we consider the summation
\[
	\sum_{ij\in E} \omega_{ij} \tilde{\omega}_{ij}
\]
Because of the closeness and the co-closeness, the summation can be rewritten as the product of integrals along the boundary of a fundamental domain, which is analogous to Stokes' theorem.

\begin{proposition}\cite{bobenko2016} \label{prop:stokes}
	Suppose $\omega$ is a closed 1-form on the primal decomposition $(V,E,F)$ with periods $(A,B)$ and $\tilde{\omega}$ is co-closed, i.e. a closed 1-form on the dual decomposition $(V^*,E^*,F^*)$, with periods $(\tilde{A},\tilde{B})$. Then
	\[
	\sum_{ij \in E} \omega_{ij} \tilde{\omega}_{ij} = A\tilde{B} - B \tilde{A}= \{(A,B),(\tilde{A},\tilde{B})\}
	\]
\end{proposition}

For any harmonic 1-form $\omega$ and $\tilde{\omega}$ on the primal graph, we consider the product
\[
\sum_{ij \in E} c_{ij} \omega_{ij} \tilde{\omega}_{ij} = 	\sum  \omega\, {*}\tilde{\omega} = \sum *\omega \, \tilde{\omega}
\]
Since $\omega$ is closed on the primal graph and $*\omega:=c\omega$ is co-closed, we can apply Proposition \ref{prop:stokes}.

\begin{corollary}
	Suppose $\omega,\tilde{\omega}:\vec{E} \to \mathbb{R}$ are harmonic 1-forms on the primal decomposition with periods $(A,B),(\tilde{A},\tilde{B})$ respectively. Then
	\begin{align} \label{eq:bilinearform}
	\sum_{ij \in E} c_{ij} \omega_{ij} \tilde{\omega}_{ij} = \{(A,B),L(\tilde{A},\tilde{B})\} = -\{L(A,B), (\tilde{A},\tilde{B})\}
	\end{align}
\end{corollary}
\begin{proof}
	We have from Proposition \ref{prop:stokes}
	\begin{align*}
		\{(A,B),L(\tilde{A},\tilde{B})\} &= \sum  \omega_{ij}\, {*}\tilde{\omega}_{ij}  \\&= \sum c_{ij} \omega_{ij} \tilde{\omega}_{ij} \\&=  \sum  \tilde{\omega}_{ij}\, {*}\omega_{ij} \\ &=\{(\tilde{A},\tilde{B}),L(A,B)\} = -\{L(A,B), (\tilde{A},\tilde{B})\}
	\end{align*}
and the claim follows.
\end{proof}

Now we are able to derive the operator $L$.
 
\begin{proposition}\label{prop:harmonicaction}
	For any edge weight $c:E \to \mathbb{R}_{>0}$, in terms of the standard basis of $\mathbb{R}^2$, the operator $L$ over the period space has the matrix form
	\begin{align*}
		L = \left( \begin{array}{cc}
			k_c \frac{\Re \tau_c}{\Im \tau_c} & -k_c\frac{1}{\Im \tau_c} \\ k_c\frac{|\tau_c|^2}{\Im \tau_c} & -k_c\frac{\Re \tau_c}{\Im \tau_c}
		\end{array} \right)
	\end{align*}
for some $k_ c >0$ and $\tau_c \in \mathbb{H}$ in the upper half plane. Furthermore, the matrix has eigenvalues $-k_c \sqrt{-1}$ and $k_c \sqrt{-1}$ with corresponding eigenvectors $\begin{pmatrix}
	1 \\ \tau_c
\end{pmatrix}$ and $\begin{pmatrix}
1\\ \bar{\tau}_c
\end{pmatrix}$.
\end{proposition}

\begin{proof}
	In terms  of the standard basis of $\mathbb{R}^2$, we write the operator $L$ as a $2{\times}2$-matrix
	\[
	L = \left( \begin{array}{cc} a & b \\c & d \end{array} \right).
	\]
	Equation \eqref{eq:bilinearform} implies for any column  vectors $U,V \in \mathbb{R}^2$, 
	\[
	U^t \left( \begin{array}{cc}
		0 & 1\\ -1 & 0
	\end{array}\right) LV = -U^t L^t \left( \begin{array}{cc}
	0 & 1\\ -1 & 0
\end{array}\right) V 
	\]
	and we deduce that $a=-d$.
	
	Furthermore, since  the energy is always non-negative, Equation \eqref{eq:bilinearform}  implies for any nonzero column vector $U \in \mathbb{R}^2$,
	\begin{align*}
		0< U^t \left( \begin{array}{cc}
			0 & 1\\ -1 & 0
		\end{array}\right) L U = U^t \left( \begin{array}{cc}
		c & -a\\ -a & -b
	\end{array}\right) U
	\end{align*}
It is deduced that $\left( \begin{array}{cc}
	c & -a\\ -a & -b
\end{array}\right)$ is positive definite. Thus
\begin{align*}
	\det \left( \begin{array}{cc}
		c & -a\\ -a & -b
	\end{array}\right) = -bc -a^2 > 0, \quad
c>0, \quad
b<0
\end{align*}
We define
\begin{align*}
 k_c:=&  \sqrt{-bc -a^2} > 0 \\
	\Re \tau_c: =& -\frac{a}{b} \\
	\Im \tau_c: =& - \frac{k}{b} >0
\end{align*}
and obtain the matrix form of $L$. One can check that
\[
L \begin{pmatrix}
	1  \\ \tau_c
\end{pmatrix} =  -k_c \sqrt{-1} \begin{pmatrix}
1  \\ \tau_c
\end{pmatrix} \quad \text{and }  \quad L \begin{pmatrix}
	1  \\ \bar{\tau}_c 
\end{pmatrix} = k_c \sqrt{-1} \begin{pmatrix}
1\\ \bar{\tau}_c 
\end{pmatrix}.
\]
\end{proof}

\section{Proof of Theorem \ref{thm}}

We investigate harmonic maps to the torus from the information of the operator $L$. Throughout the section, we assume $\tau_c$, $k_c$ to be the one given in Proposition \ref{prop:harmonicaction}. 

\begin{proposition}\label{prop:minen}
 Let $\tau \in\mathbb{H}$ represent any Euclidean torus and $h_{\tau}:V \to S_{\tau}$ be the corresponding harmonic map. Then its energy satisfies
 \[
 \mathcal{D}_c(\tau,h_{\tau})= D_c(h_{\tau}) \geq k_c = \mathcal{D}_c(\tau_c,h_{\tau_c})
 \]
 and $(\tau_c,h_{\tau_c})$ is the unique minimizer  of $\mathcal{D}_c$.
\end{proposition}
\begin{proof}
	We abbreviate $h_{\tau}$ as $h$ and write $\tilde{h}$ the lift of $h$ to the universal cover. For every oriented edge $ij$, we consider
	\[
	\omega_{ij}:= \Re (\tilde{h}_{\tilde{j}}-\tilde{h}_{\tilde{i}}),  \quad  \eta_{ij}:= \Im (\tilde{h}_{\tilde{j}}-\tilde{h}_{\tilde{i}}).
	\]
	which defines harmonic 1-forms on the torus $(V,E,F)$. Equation \eqref{eq:htperiods} implies that they have periods
	\begin{align*}
		\sum_{\gamma_1} \omega =  \frac{1}{\sqrt{\Im \tau}}, \quad & \sum_{\gamma_2} \omega =  \frac{\Re \tau}{\sqrt{\Im \tau}} \\
		\sum_{\gamma_1} \eta =  0, \quad & \sum_{\gamma_2} \eta =  \frac{\Im \tau}{\sqrt{\Im \tau}} 
	\end{align*}
	We compute its energy using Proposition \ref{prop:stokes}
	\begin{align*}
		D_c(h_{\tau})&=\frac{1}{2} \left( \sum_{ij} c_{ij} \omega_{ij}^2 +  \sum_{ij} c_{ij} \eta_{ij}^2\right) \\
		&=\frac{1}{2}\{( \frac{1}{\sqrt{\Im \tau}},\frac{\Re \tau}{\sqrt{\Im \tau}}),L( \frac{1}{\sqrt{\Im \tau}},\frac{\Re \tau}{\sqrt{\Im \tau}})\} + \frac{1}{2} \{(0,\frac{\Im \tau}{\sqrt{\Im \tau}} ),L(0,\frac{\Im \tau}{\sqrt{\Im \tau}} )\}\\
		&= \frac{k_c}{2 \Im  \tau_c \Im \tau} ( |\tau_c|^2 + |\tau|^2 - 2 \Re \tau_c \Re \tau) \\
		& =  \frac{k_c}{2 \Im  \tau_c \Im \tau} ( (\Re \tau_c - \Re \tau )^2 + |\Im \tau_c|^2 + |\Im \tau|^2 ) \\
			& =  \frac{k_c |\tau-\tau_c|^2}{2 \Im  \tau_c \Im \tau}  + k_c\\
		& \geq k_c
	\end{align*}
	The equality holds if and only if $\tau=\tau_c$.  
\end{proof} 

\begin{proposition}\label{prop:reci}
	For  any $\tau \in \mathbb{H}$, the conjugate map of $h_{\tau}:V \to S_{\tau}$ projects to $k S_{\tau}$ for some $k>0$ if and only if $(\tau,h_{\tau})=(\tau_c, h_{\tau_c})$ and $k=k_c$. 
\end{proposition}
\begin{proof}
	We abbreviate $h_{\tau}$ as $h$ and define $\omega,\eta$ as in the proof of Proposition \ref{prop:minen}. We consider the conjugates $*\omega$ and $*\eta$ which have periods
	\begin{align*}
		(\sum_{\gamma_1} *\omega,  \sum_{\gamma_2} *\omega) &= L( \frac{1}{\sqrt{\Im \tau}},\frac{\Re \tau}{\sqrt{\Im \tau}}) = \frac{k_c}{\Im \tau_c \sqrt{\Im \tau}}( \Re \tau_c - \Re \tau, |\tau_c|^2 - \Re \tau_c \Re \tau) \\
			(\sum_{\gamma_1} *\eta,  \sum_{\gamma_2} *\eta) &= L( 0,\frac{\Im \tau}{\sqrt{\Im \tau}}) = \frac{k_c}{\Im \tau_c \sqrt{\Im \tau}}(  - \Im \tau, - \Re \tau_c \Im \tau)
	\end{align*}
	By Equation \eqref{eq:conharmmap}, the conjugate harmonic map  $\tilde{h}^*: \tilde{V} \to \mathbb{C}$ is in the form
	\[
	\tilde{h}^*_{\mbox{left}(ij)} -\tilde{h}^*_{\mbox{right}(ij)}  = -*\eta_{ij} + \sqrt{-1} * \omega_{ij}
	\]
	and hence
	\begin{align}
		\tilde{h}^*\circ \gamma_1 - \tilde{h}^*&= \frac{k_c}{\sqrt{\Im \tau_c}}( \sqrt{\frac{ \Im \tau}{\Im \tau_c}} + \sqrt{-1} \,\frac{\Re \tau_c -\Re \tau}{\sqrt{\Im \tau_c \Im \tau}}) \label{eq:1} \\
		\tilde{h}^*\circ \gamma_2 - \tilde{h}^* &= \frac{k_c}{\sqrt{\Im \tau_c}}  ( \frac{\sqrt{\Im \tau} \Re \tau_c}{\sqrt{\Im \tau_c}} + \sqrt{-1} \, \frac{|\tau_c|^2 - \Re \tau_c \Re \tau}{\Im \tau_c}) \label{eq:2}
	\end{align}
On the other hand, we look for $\tau \in \mathbb{H}$ such that $\tilde{h}^*$ projects to $kS_{\tau}$ for some $k>0$, i.e. in the form
	\begin{align}
	\tilde{h}^*\circ \gamma_1 - \tilde{h}^*&=  \frac{k}{\sqrt{\Im \tau}}  \label{eq:3}\\
	\tilde{h}^*\circ \gamma_2 - \tilde{h}^* &= \frac{k \tau}{\sqrt{\Im \tau}} \label{eq:4} 
\end{align}
Comparing the imaginary part of \eqref{eq:1} and \eqref{eq:3} yields $\Re \tau=\Re \tau_c$. Comparing the real part of \eqref{eq:1} and \eqref{eq:3} implies $k_c \Im \tau = k \Im \tau_c$.  Substituting them into \eqref{eq:2} and \eqref{eq:4}, we obtain $\tau=\tau_c$ and $k=k_c$.
\end{proof}

We now relate the optimal Euclidean structure with a weighted Delaunay decomposition.

\begin{proof}[Proof of Theorem \ref{thm}]
	Statements (i) $\iff$ (ii) follow from Proposition \ref{prop:minen} and \ref{prop:reci}.
	
	To show (ii) $\implies$ (iii). Suppose $(\tau_c,h_{\tau_c})$ is the given harmonic map whose scaled conjugate map $h_{\tau_c}^{\dagger}:= h^*_{\tau_c}/k_{c}$ projects to the same torus as $h_{\tau_c}$. The map $h^{\dagger}_{\tau_c}$ is unique up to translation. From the result in \cite[Theorem 4.4]{Lin2020}, there is a unique choice for $h^{\dagger}_{\tau_c}$ such that $h_{\tau_c}$ is a realization of a weighted Delaunay decomposition and $h^{\dagger}_{\tau_c}$ is the dual Voronoi diagram. For the edge weight, it is shown by an index argument in \cite{Gortler2006} that edges have non-zero lengths. Hence, the right hand side of Equation \eqref{eq:dualprimal} is well defined and its $k_c$-multiple is equal to the edge weight $c$ by the construction of the conjugate map. 
		
	To show (iii) $\implies$ (i). Suppose $h:(V,E) \to S_{\tau}$ is induced from a weighted Delaunay decomposition and  $h^{\dagger}$ is the dual Voronoi diagram. For any $k>0$, we define the edge weight $c$ to be $k$-multiple of the canonical weight (eq. \eqref{eq:dualprimal}). The edge weight is real-valued since the primal edge under $h$ is perpendicular to the dual edge realized by $h^{\dagger}$. Indeed, it takes positive values because of the Delaunay condition. Observe that the scaled map $h^*:= k h^{\dagger}$ is the conjugate to $h$ mapping the dual graph to the scaled torus $k  S_{\tau}$. By Proposition \ref{prop:minen} and \ref{prop:reci}, we know $(\tau, h)$ is the minimizer of $\mathcal{D}_c$ with value $k$.
\end{proof}

\begin{remark}
	In fact, the Dirichlet energy associated to a weighted Delaunay decomposition is the surface area. Suppose $h:(V,E) \to S_{\tau}$ is induced from a weighted Delaunay decomposition and  $h^{\dagger}$ is the dual Voronoi diagram. We take the canonical edge weight $c$ in Equation \eqref{eq:dualprimal}. Then for every edge $ij$, the quantity
	\[
	\frac{1}{2} c_{ij} \ell_{ij}^2 = \frac{1}{2}	|\tilde{h}^{\dagger}_{\mbox{left}(ij)} -\tilde{h}^{\dagger}_{\mbox{right}(ij)}| \cdot |\tilde{h}_j -\tilde{h}_i|
	\]
	is the area of the quadrilateral formed by the four vertices $\tilde{h}^{\dagger}_{\mbox{left}(ij)}, \tilde{h}_i,  \tilde{h}^{\dagger}_{\mbox{right}(ij)}, \tilde{h}_j$ (See figure \ref{fig:kite}). Hence the energy
	\[
		D_c(h) = \frac{1}{2} \sum_{ij \in E} c_{ij} \ell_{ij}^2 = \area(S_{\tau}).
	\]
\end{remark}
	\begin{figure}[h]
	\includegraphics[width=0.8\textwidth]{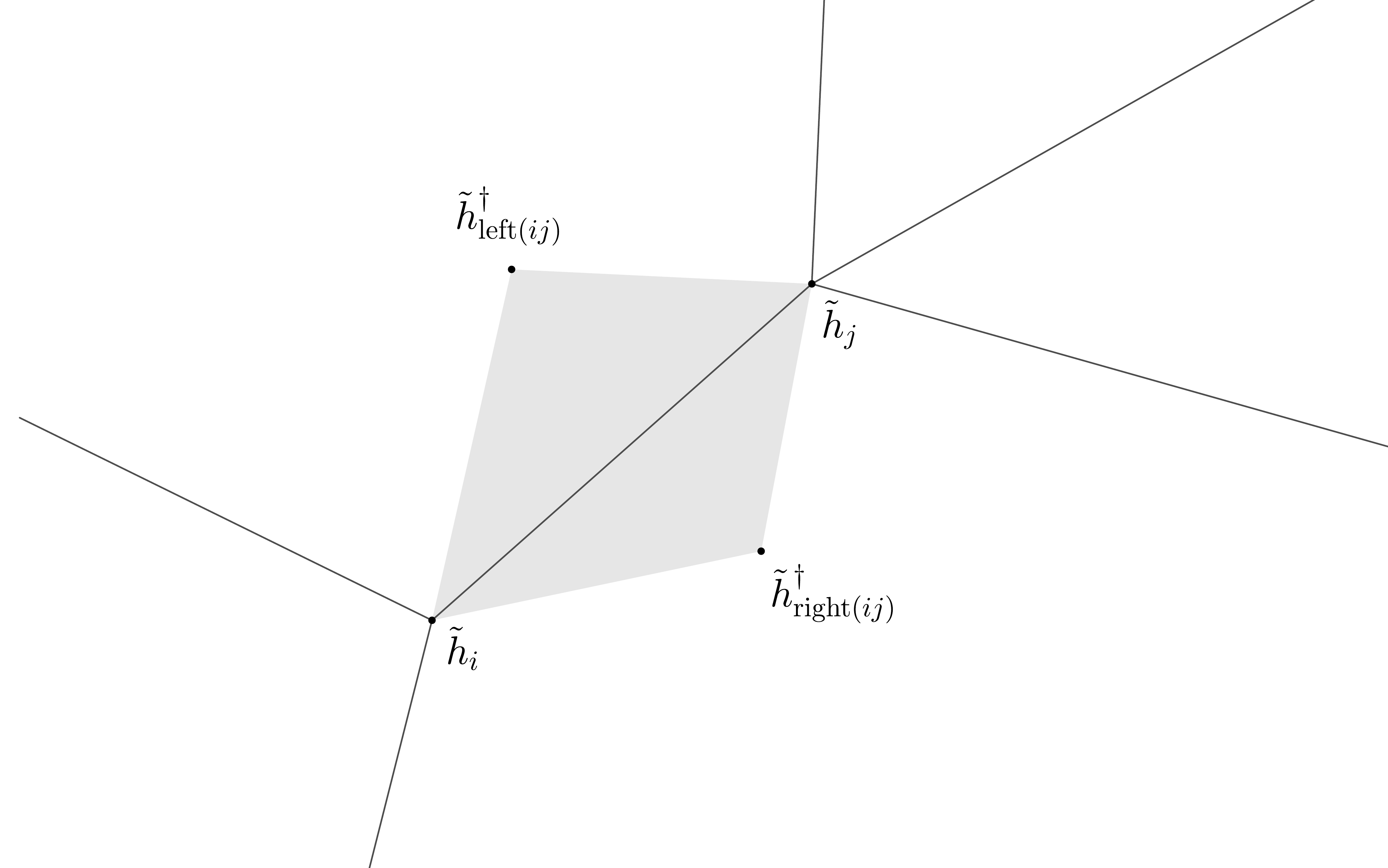}
	\caption{When the edge weight is induced from a weighted Delaunay decomposition, the energy associated to the edge $ij$ is the area of the shaded quadrilateral.}
	\label{fig:kite}
\end{figure}

\begin{remark}
	When minimizing the Dirichlet energy over Euclidean tori, it is important to normalize the tori to have unit area. If one considers discrete harmonic maps to Euclidean tori generated by $1$ and $\tau$ in the complex plane without normalization, one can modify the proof of Theorem \ref{prop:minen} to show that the Dirichlet energy decreases whenever $\Im \tau$ decreases. The Dirichlet energy achieve the minimum when $\Im \tau=0$. In this case, the Euclidean torus degenerates to a line.
\end{remark}

\section{Response matrix $L$ from Laplace operator $\Delta$}

In the section, we shall express the response matrix $L:\mathbb{R}^2 \to \mathbb{R}^2$ in terms of the discrete Laplacian and prove Theorem \ref{thm:energy}.

We first define an incidence matrix $d$. We fix an arbitrary orientation for every edge $e$ so that $e_{+}$ and $e_{-}$ represent the head and the tail of the oriented edge $e$. Then we define $d: \mathbb{R}^{V} \to \mathbb{R}^E$ by
\[
df(e) = f(e_{+}) - f(e_{-}). 
\]
We further define
\[
\Delta := d^{T} Cd
\]
where $C:\mathbb{R}^E \to \mathbb{R}^E$ is the diagonal matrix of the corresponding positive edge weights and $d^T$ is the transpose of $d$. One can show that the operator $\Delta : \mathbb{R}^V \to \mathbb{R}^V$ satisfies for every vertex $i$
\[
(\Delta f)_i = \sum_j c_{ij} (f_j-f_i)
\]
and $\Delta$ is the so called discrete Laplacian. It is positive semi-definite and the kernel consists of constant functions.

Notice that each column vector of $d$ represents a closed discrete 1-form that is exact. Any $|V|-1$ column vectors form a basis of the space of exact 1-forms on the cell decomposition $G=(V,E,F)$. To get a basis of the space of closed 1-forms, we consider two more closed 1-forms $m_1$ and $m_2$ that have nontrivial periods
\begin{gather*}
	\sum_{\gamma_1} m_1=1, \sum_{\gamma_2} m_1=0\\
	\sum_{\gamma_1} m_2=0, \sum_{\gamma_2} m_2=1
\end{gather*}
Thus, every closed 1-form with periods
\[
\sum_{\gamma_1} \omega=A, \sum_{\gamma_2} \omega=B
\]
can be expressed as
\[
\omega = d f +  \begin{pmatrix}
	\vert & \vert  \\
	m_1 & m_2 \\
		\vert & \vert
\end{pmatrix} \begin{pmatrix}
A \\ B
\end{pmatrix} =:  d f + M \begin{pmatrix}
A \\ B
\end{pmatrix} 
\]
for some $f \in \mathbb{R}^{V}$ unique up to constants. The closed 1-form $\omega$ is harmonic if
\[
0 = d^T C \omega =  d^T C d f +  d^T C M \begin{pmatrix}
	A \\ B
\end{pmatrix} = \Delta f +  d^T C M \begin{pmatrix}
A \\ B
\end{pmatrix}.
\]
To obtain a solution in terms of $f$ uniquely, we can pick a vertex $o$ and demand $f_o=0$. We write $d_{\bar{o}}$ as the submatrix of $d$ with the column corresponding to vertex $o$ removed. We also write $\Delta_{\bar{o} \bar{o}}$ as the submatrix of $\Delta$ with the column and the row corresponding to vertex $o$ removed. One can show that $\Delta_{\bar{o} \bar{o}}$ is invertible. The values of $f$ at vertices other than $o$ can be obtained via
\begin{align}\label{eq:inversef}
	f_{\bar{o}} = - \Delta_{\bar{o} \bar{o}}^{-1} d_{\bar{o}}^T C M \begin{pmatrix}
		A \\ B
	\end{pmatrix}.
\end{align}
Now we can compute the operator $L: \mathbb{R}^2 \to \mathbb{R}^2$ in terms of the edge weights. 
\begin{proposition}\label{prop:Lop}
	\[
	L = \begin{pmatrix}
		0 & -1 \\ 1 & 0
	\end{pmatrix}  (- M^T C d_{\bar{o}} \Delta_{\bar{o} \bar{o}}^{-1} d_{\bar{o}}^T C M + M^T CM )
	\]
\end{proposition}
\begin{proof}
	To compute the formula, we need further notations on the dual graph. Since the orientation of the primal edges is chosen, it naturally induces an orientation of the dual edges. Namely, a dual edge $*e$ is oriented from the right face of $e$ to the left. We then define the incidence matrix $d_*: \mathbb{R}^F \to \mathbb{R}^E$ similarly. The columns of $d_*$ span the space of exact 1-forms on the dual graph. We further define a $|E|{\times}2$ matrix $M_*$ such that its columns $m_{*1}$ and $m_{*2}$ represent closed 1-forms on the dual graph having nontrivial periods 
	\begin{gather*}
		\sum_{\gamma_1} m_{*1}=1, \sum_{\gamma_2} m_{*1}=0\\
		\sum_{\gamma_1} m_{*2}=0, \sum_{\gamma_2} m_{*2}=1
	\end{gather*}
	For a harmonic 1-form $\omega$, it is closed and $*\omega$ is co-closed. So we have
	\[
	C d f + CM \begin{pmatrix}
		A \\ B
	\end{pmatrix}  = C \omega = d_* g + M_* \begin{pmatrix}
		A^* \\ B^*
	\end{pmatrix}
	\]
	for some $f:V \to \mathbb{R}$ and $g:F \to \mathbb{R}$. Applying $M^T$ to both sides from the left yields
	\[
	M^T C d f + M^T CM \begin{pmatrix}
		A \\ B
	\end{pmatrix}  = (M^T d_*) g + M^T M_* \begin{pmatrix}
		\tilde{A} \\ \tilde{B}
	\end{pmatrix} = 0 + \begin{pmatrix}
		0 & 1 \\ -1 & 0
	\end{pmatrix}  \begin{pmatrix}
		A^* \\ B^*
	\end{pmatrix}
	\]
	where Proposition \ref{prop:stokes} is applied to  $M^T d_*$ and $M^T M_*$ since the rows of $M^T$ represent closed 1-forms on the primal  graph while the columns of $d_*$ and $M_*$ represents co-closed 1-forms. Because $f_{o}=0$, we have from Equation \eqref{eq:inversef}
	\begin{align*}
		\begin{pmatrix}
			A^* \\ B^*
		\end{pmatrix} &=  \begin{pmatrix}
			0 & -1 \\ 1 & 0
		\end{pmatrix}  (M^T C d_{\bar{o}} f_{\bar{o}} + M^T CM \begin{pmatrix} 
			A \\ B
		\end{pmatrix})  \\  &= \begin{pmatrix}
			0 & -1 \\ 1 & 0
		\end{pmatrix}  (- M^T C d_{\bar{o}} \Delta_{\bar{o} \bar{o}}^{-1} d_{\bar{o}}^T C M + M^T CM ) \begin{pmatrix} 
			A \\ B
		\end{pmatrix}
	\end{align*}
	and thus 
	\[
	L = \begin{pmatrix}
		0 & -1 \\ 1 & 0
	\end{pmatrix}  (- M^T C d_{\bar{o}} \Delta_{\bar{o} \bar{o}}^{-1} d_{\bar{o}}^T C M + M^T CM ).
	\]
\end{proof}

We now express the minimal energy in terms of the discrete Laplacian.

\begin{proof}[Proof of Theorem \ref{thm:energy}]
	From Proposition \ref{prop:harmonicaction}, the minimal energy is $k_c$. We know
	\begin{equation} \label{eq:ener}
		k^2_c = \det L = \det    ( M^T CM  - M^T C d_{\bar{o}} \Delta_{\bar{o} \bar{o}}^{-1} d_{\bar{o}}^T C M)
	\end{equation}
	Recall that the determinant of a block matrix satisfies
	\[
	\det \begin{pmatrix}
		P & Q \\
		Q^T & R 
	\end{pmatrix} = \det  P \det(R -Q^T P^{-1} Q)
	\]
	whenever $P$ is an invertible matrix. Applying this formula to Equation \eqref{eq:ener} yields
	\[
	k^2_c = \frac{\det \begin{pmatrix}
			\Delta_{\bar{o} \bar{o}} &  d_{\bar{o}}^T C M \\ M^T C d_{\bar{o}} & M^T CM
	\end{pmatrix}}{\det  	\Delta_{\bar{o} \bar{o}}  } =  \frac{\det (\tilde{d}^T_{\bar{o}} C \tilde{d}_{\bar{o}})}{\det ( d_{\bar{o}}^T C d_{\bar{o}})}
	\]
	where 
	\[
	\tilde{d}_{\bar{o}} = \begin{pmatrix}
		d_{\bar{o}} & M 
	\end{pmatrix}
	\]
	is a  $|E|{\times}(|V|{+}1)$ matrix. Notice that both matrices $\tilde{d}^T C \tilde{d}$ and $d^T C d$ are positive semi-definite. One has
	\[
	d^T C df =0
	\]
	if and only if $f \in \mathbb{R}^V$ is a constant function. On the other hand
	\[
	\tilde{d}^T C \tilde{d} \begin{pmatrix}
		f \\ A \\ B 
	\end{pmatrix} =0
	\]
	if and only if $f \in \mathbb{R}^{V}$  is a constant function while $A=B=0$. Thus
	\[
	k = \sqrt{  \frac{\det (\tilde{d}^T_{\bar{o}} C \tilde{d}_{\bar{o}})}{\det ( d_{\bar{o}}^T C d_{\bar{o}})} } = \sqrt{ \frac{\det_0 (\tilde{d}^T C \tilde{d})}{\det_0( d^T C d)}}
	\]
	and $\det_0$ denotes the determinant of a matrix with the row and column corresponding to the vertex $o$ removed.
\end{proof}
\begin{remark}
	The discrete Laplace operator $ d^T C d$ is a sub-matrix of  $\tilde{d}^T C \tilde{d}$. The Schur complement of the block $(d^T C d)_{\bar{o}\bar{o}}$ of the matrix $(\tilde{d}^T C \tilde{d})_{\bar{o}\bar{o}}$ is the matrix
	\[
	 \begin{pmatrix}
		0 & 1 \\ -1 & 0
	\end{pmatrix} L
	\]
	where $L$ is the response matrix that we used.
\end{remark}

\begin{example}\label{example:4}
	We illustrate how to compute the incidence matrix. Consider a cell decomposition of a torus as in Figure \ref{fig:4vertex}. There are 4 vertices and 8 edges. The edges are oriented as described in the caption and labeled as $e_j$ with corresponding conductance $c_j$, where $j=1,2,\dots 8$. The incidence matrix $d$ has size $|V|$ by $|E|$. The $(i,j)$ entry corresponds to
	\[
	df^{(i)}(e_j)= f^{(i)}(e_{j+}) - f^{(i)}(e_{j-}).
	\]
	where $f^{(i)}:V \to \mathbb{R}$ takes value 1 at vertex $i$ and 0 otherwise. The symbols $e_{j+}$ and $e_{j-}$ indicate the vertices that is the head and the tail of the oriented edge $e_j$. One gets
	\[
	d= \left(
	\begin{array}{cccc}
		-1 & 1 & 0 & 0 \\
		0 & -1 & 0 & 1 \\
		0 & 0 & -1 & 1 \\
		-1 & 0 & 1 & 0 \\
		1 & -1 & 0 & 0 \\
		0 & 0 & 1 & -1 \\
		1 & 0 & -1 & 0 \\
		0 & 1 & 0 & -1 \\
	\end{array}
	\right),
	\]
	where the columns span the space of exact 1-forms. To get closed 1-forms that are not exact, consider a vertical line oriented homotopic to $\gamma_2$ not passing through any vertex, say the dotted vertical line. It intersects with the oriented edges $e_5$ and $e_6$ once positively. We then define a 1-form $m_1$ which is equal to 1 along $e_5$ and $e_6$ but zero otherwise. Then we can check $m_1$ is a closed 1-form with non-vanishing periods
	\[
		\sum_{\gamma_1} m_1=1,\quad \sum_{\gamma_2} m_1=0.
	\]
    Similarly, the dotted horizontal line is homotopic to $\gamma_1$ and intersect $e_7,e_8$ once positively. We then define a 1-form $m_2$ which is equal to 1 along $e_7$ and $e_8$ but zero otherwise. Then we can check $m_2$ is a closed 1-form with non-vanishing periods
    \[
    \sum_{\gamma_1} m_2=0,\quad \sum_{\gamma_2} m_2=1.
    \]
	It yields the columns of matrix $M$
	\[
	M=\left(
	\begin{array}{cc}
		0 & 0 \\
		0 & 0 \\
		0 & 0 \\
		0 & 0 \\
		1 & 0 \\
		1 & 0 \\
		0 & 1 \\
		0 & 1 \\
	\end{array}
	\right).
	\] 
	From these, one can compute the operator $L$ via Proposition \ref{prop:Lop}.
	
	\begin{figure}
			\includegraphics[width=0.4\textwidth]{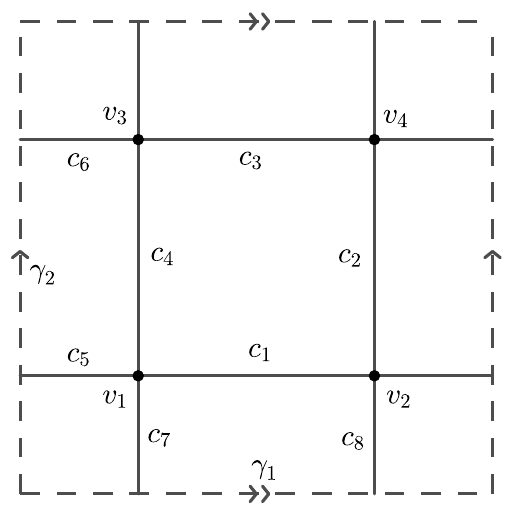}
		\caption{A cell decomposition of the torus with 4 vertices. All horizontal edges are oriented to the right. All vertical edges are oriented upward.}
		\label{fig:4vertex}
	\end{figure}
\end{example}

We consider the simplest example where we have space to write down the operator $L$.

\begin{example} \label{example}
	We investigate a one-vertex triangulation of a Euclidean torus as in Figure \ref{fig:1vertex} (Left). It is essentially 3-connected since its lift to the universal cover is the triangular lattice which is 3-connected. We pick an orientation of edges as indicated in the figure. Since for every edge the head and the tail coincide, we have the incidence matrix vanishing. We also obtain the matrix $M$ by considering the intersection with a vertical line and a horizontal line respectively as in Example \ref{example:4}. Altogether we have
	\begin{align*}
		d=\left(
		\begin{array}{c}
			0 \\
			0 \\
			0 \\
		\end{array}
		\right), \quad M=\left(
		\begin{array}{cc}
			 1 & 0 \\
			1 & 1 \\
			 0 & 1 \\
		\end{array}
		\right), 
	\end{align*}
	and obtain via Proposition \ref{prop:Lop}
	\[
	L=\left(
	\begin{array}{cc}
		-c_2 & -c_2-c_3 \\
		c_1+c_2 & c_2 \\
	\end{array}
	\right).
	\]
    One finds that the minimal energy is
	\[
	k_c = \sqrt{\det(L)}= \sqrt{c_1 c_2 + c_2 c_3 + c_3 c_1}
	\]
	and the optimal Euclidean structure is
	\[
	\tau_c = \frac{- c_2 + \mathbf{i}\sqrt{c_1 c_2 + c_2 c_3 + c_3 c_1}}{c_2 +c_3}  
	\]
	where $\mathbf{i}=\sqrt{-1}$. Observe that $\Re \tau_c <0$ for this weighted graph. In other words,  with this combinatoric, any realization to a Euclidean torus $S_{\tau}$ with $\Re \tau \geq0$ cannot be the minimizer of the energy $\mathcal{D}_c$ for any positive edge weights $c$.
	\begin{figure}[h]
	\includegraphics[width=0.4\textwidth]{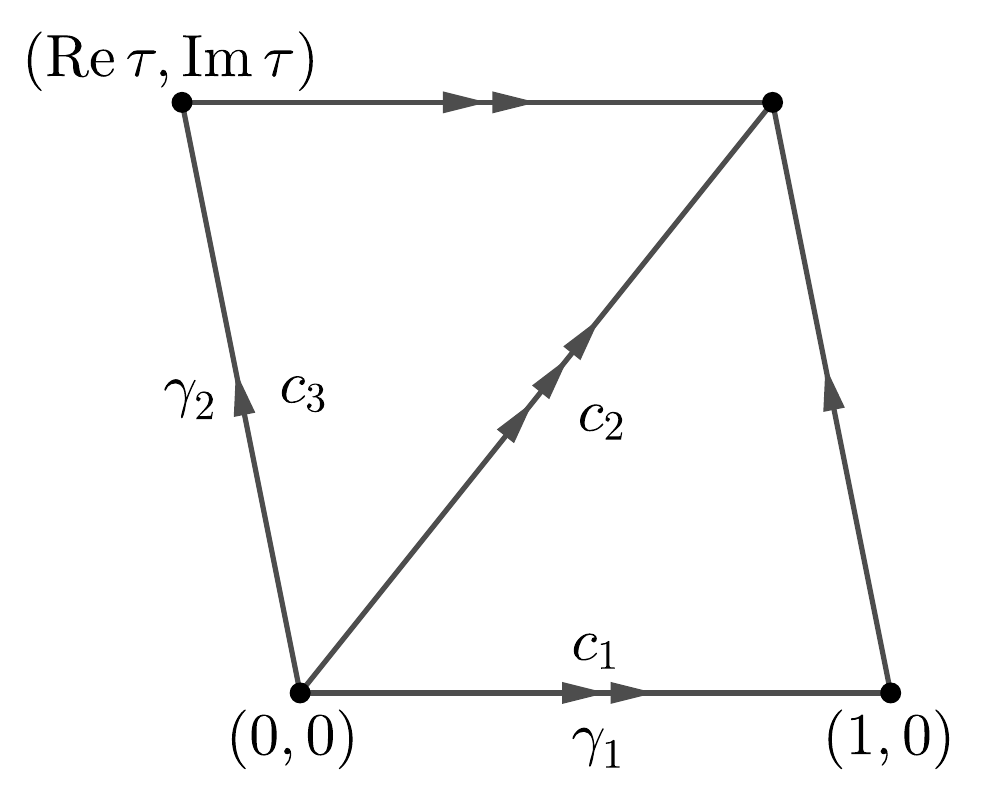}
	\includegraphics[width=0.4\textwidth]{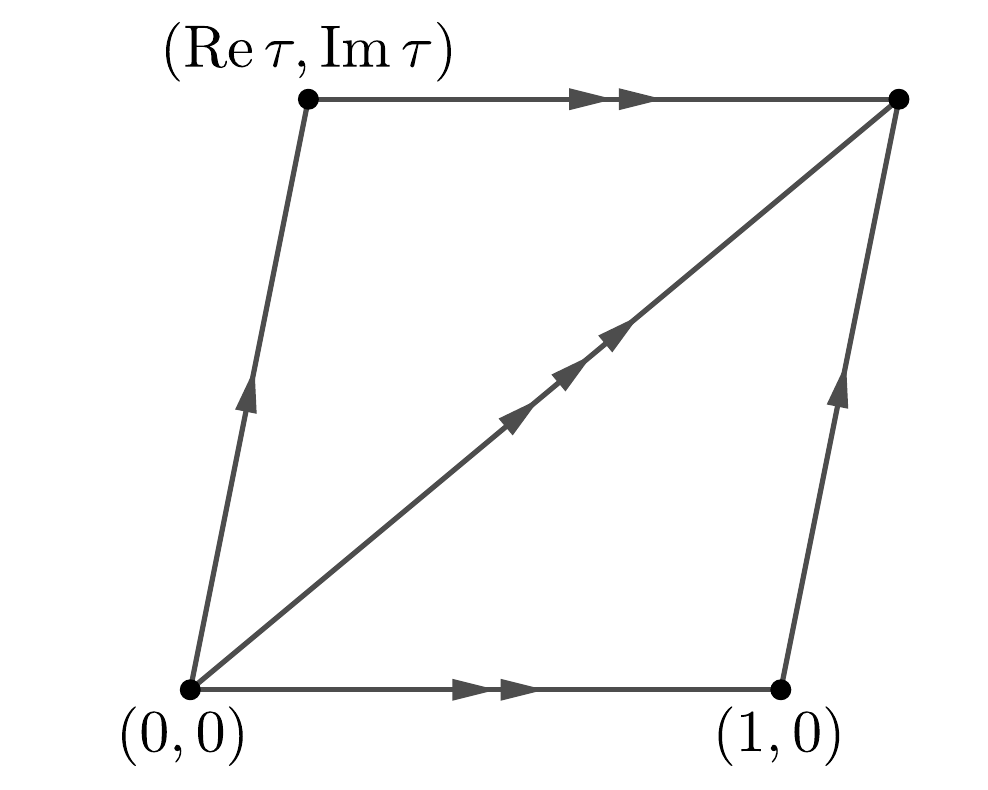}
		\caption{A one-vertex triangulation of the torus is shown on the left. Opposite sides are identified. The orientation of the edges are indicated by the arrows. The one-vertex triangulation on the right cannot be the minimize of $\mathcal{D}_c$, since the Euclidean torus has $\Re \tau >0$.}
		\label{fig:1vertex}
	\end{figure}
\end{example}

\begin{example}
	There are many more straight-line embeddings that could not be minimizers of the Dirichlet energy $\mathcal{D}_c$ for any choice of positive edge weight $c$. Figure \ref{fig:1vertexwrap} shows another one-vertex triangulation of the square torus  with a long diagonal wrapping around the torus. Here the square torus has $\tau=\sqrt{-1}$, particularly $\Re \tau=0$. Repeating the computation in Example \ref{example} shows that the optimal Euclidean torus has $\Re \tau_c <0$. Thus, with the vertices fixed on the square torus, this combinatoric can not be induced from any weighted Delaunay decomposition. 
	\begin{figure}[h]
		\includegraphics[width=0.28\textwidth]{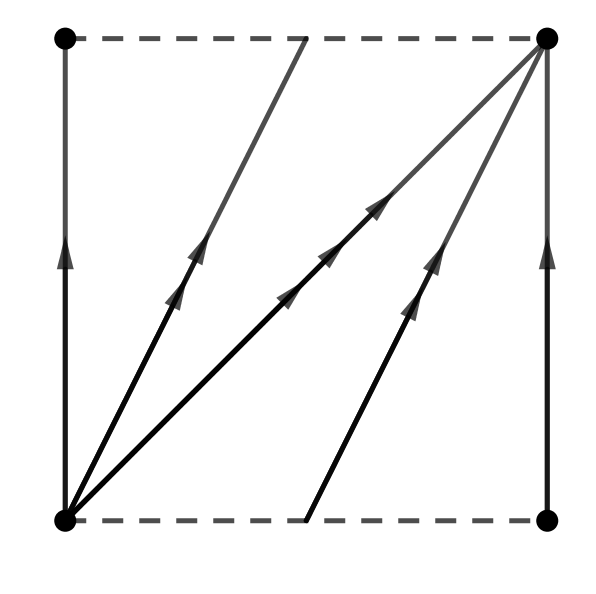}
		\caption{A one-vertex triangulation of the square torus that is not a minimizer of  $\mathcal{D}_c$ for any choice of positive edge weight $c$.}
		\label{fig:1vertexwrap}
	\end{figure}
\end{example}

\section{Weights with arbitrary sign}\label{sec:sign}

In Theorem \ref{thm}, the assumption about edge weights being positive is not compulsory. We can allow the edge weights taking negative values as long as the energy functional is positive definite over the space of closed discrete 1-forms.

\begin{definition}\label{def:energy}
	Let $(V,E,F)$ be a cellular decomposition of a torus. The edge weight $c:E \to \mathbb{R}$ is said to be non-degenerate if for every closed 1-form $\omega:\vec{E} \to \mathbb{R}$
\[
D_c(\omega)= \frac{1}{2}\sum_{ij} c_{ij} \omega_{ij\in E}^2 \geq 0
\]
and the equality holds if and only if $\omega \equiv 0$.
\end{definition}

\begin{proposition}
Suppose the edge weight $c:E \to \mathbb{R}$ is non-degenerate. Then over the period space $L:\mathbb{R}^2 \to \mathbb{R}^2$ is well defined. 
\end{proposition}
\begin{proof}
	We first show that if $\omega$ is a harmonic 1-form with vanishing periods, then it must be trivial. Observe that  $\omega$ having vanishing periods implies there exists $f:V \to \mathbb{R}$ such that $\omega_{ij}= f_j - f_i$. Thus
	\[
	D_c(\omega)= \frac{1}{2}\sum_{ij\in E} c_{ij} \omega_{ij}^2 = \sum_{i\in V} f_i \left( \sum_j c_{ij} (f_j - f_i) \right)=0 
	\]
	because of the harmonicity
	\[
	0= \sum_j c_{ij} \omega_{ij} = \sum_j c_{ij} (f_j - f_i) \quad \forall i \in V.
	\]
	Thus, the non-degeneracy of the edge weight yields $\omega \equiv 0$. It implies the uniqueness of the harmonic 1-form with any given periods if exist. Indeed, using a dimension argument, we further deduce that for any prescribed period $(A,B)$, there exists a unique harmonic 1-form $\omega$ such that
	\[
	\sum_{\gamma_1} \omega = A,  \quad 	\sum_{\gamma_2} \omega = B
	\]
	Thus the map $L:\mathbb{R}^2 \to \mathbb{R}^2$ as in Section \ref{sec:conjugate} is well defined. 
\end{proof}

For non-degenerate edge weights, one can check Proposition \ref{prop:harmonicaction}, \ref{prop:minen} and \ref{prop:reci} hold and the proofs remain the same. We have a generalization of Theorem \ref{thm}.

  \begin{theorem}\label{ethm}

	Let $(V,E,F)$ be a cellular decomposition of a topological torus and $c:E \to \mathbb{R}$ be non-degenerate edge weight.  Then the following statements on Euclidean structure $\tau$, harmonic map $h_{\tau}$ as well as a positive number $k$ are equivalent:  
\begin{enumerate}[(i)]
	\item $(\tau, h_{\tau})$ is the unique minimizer of $\mathcal{D}_c$ with value $k$.
	\item The conjugate map of harmonic map $h_{\tau}:V \to S_{\tau}$ projects to the same torus $S_{\tau}$ up to scaling $k>0$.
\end{enumerate}
\end{theorem}

In this case, the harmonic map generally does not gives a Delaunay decomposition since the edge weight might take negative values. 

\section*{Acknowledgment}

The author would like to thank Toru Kajigaya, Richard Kenyon, Benedikt Kolbe and Yanwen Luo for comments on the draft and fruitful discussions. 

	\bibliographystyle{plainurl}
	\bibliography{torus}
	
\end{document}